\documentclass[12pt,leqno]{amsart}
\usepackage{amssymb,amsthm,amsmath,latexsym}
\newcommand{\vfi}{\varphi}

\newtheorem{theorem}{\sc Theorem}[section]
\newtheorem{thm}[theorem]{\sc Theorem}
\newtheorem{lem}[theorem]{\sc Lemma}
\newtheorem{prop}[theorem]{\sc Proposition}

\newtheorem{rem}[theorem]{\sc Remark}

\newtheorem*{thmA}{Theorem A}
\newtheorem*{thmB}{Theorem B}

\title[Non-abelian tensor square]{The non-abelian 
tensor square of residually finite groups}
\author[Bastos]{R. Bastos}
\address{ Departamento de Matemat\'atica, Universidade de Bras\'ilia,
Brasilia-DF, 70910-900 Brazil }
\email{bastos@mat.unb.br}
\author[Rocco]{N.\,R. Rocco }
\address{ Departamento de Matemat\'atica, Universidade de Bras\'ilia,
Brasilia-DF, 70910-900 Brazil }
\email{norai@unb.br}

\begin{document}

\maketitle

\begin{abstract}
Let $m,n$ be positive integers and $p$ a prime. We denote by $\nu(G)$ an extension of the non-abelian tensor square $G \otimes G$ by $G \times G$. We prove that if $G$ is a residually finite group satisfying some non-trivial identity $f \equiv~1$ and for every $x,y \in G$ there exists a $p$-power $q=q(x,y)$ such that $[x,y^{\vfi}]^q = 1$, then the derived subgroup $\nu(G)'$ is locally finite (Theorem A). Moreover, we show that if $G$ is a residually finite group in which for every $x,y \in G$ there exists a $p$-power $q=q(x,y)$ dividing $p^m$ such that $[x,y^{\vfi}]^q$ is left $n$-Engel, then the non-abelian tensor square $G \otimes G$ is locally virtually nilpotent (Theorem B).        
\end{abstract}

\maketitle

\section{Introduction}The non-abelian tensor square $G \otimes G$ of a group $G$ as introduced by Brown and Loday \cite{BL84, BL} is defined to be the group generated by all symbols $\; \, g\otimes h, \; g,h\in G$, subject to the relations
\[
gg_1 \otimes h = ( g^{g_1}\otimes h^{g_1}) (g_1\otimes h) \quad
\mbox{and} \quad g\otimes hh_1 = (g\otimes h_1)( g^{h_1} \otimes
h^{h_1})
\]
for all $g,g_1, h,h_1 \in G$.

This is a particular case of a non-abelian tensor product $G \otimes H$, of groups $G$ and $H$, under the assumption that $G$ and $H$ acts one on each other and on itself by conjugation. 

We observe that the defining relations of the tensor square can be viewed as abstractions of commutator relations; thus in \cite{NR1} the following related construction is considered: Let $G$ and $H$ be 
groups and $\varphi : H \rightarrow H^{\varphi}$ an isomorphism ($H^{\varphi}$ is an isomorphic copy of $H$, where $h \mapsto h^{\varphi}$, for all $h \in H$).
 Define the group $\nu(G)$ to be \[ \nu (G):= \langle 
G,G^{\varphi} \ \vert \ [g_1,{g_2}^{\varphi}]^{g_3}=[{g_1}^{g_3},({g_2}^{g_3})^{\varphi}]=[g_1,{g_2}^{\varphi}]^{{g_3}^{\varphi}},
\; \ g_i \in G \rangle .\]

The motivation for studying $\nu(G)$ is the commutator connection:
\begin{prop} \label{prop.g} $($\cite[Proposition 2.6]{NR1}$)$ The map  $\Phi: G \otimes G \rightarrow [G, G^{\varphi}]$,
defined by $g \otimes h \mapsto [g , h^{\varphi}], \; \forall g, h \in G,$ is an isomorphism.
\end{prop}

From now on we identify the tensor square $G \otimes G$ with the subgroup $[G,G^{\vfi}]$ of $\nu(G)$ (see for instance \cite{NR} for more details). 

It is a natural task to study structural aspects of $\nu(G)$ and $G \otimes G$ for different classes of infinite groups $G.$ Moravec \cite{M} proves that if $G$ is locally finite then so is $G \otimes G$ (and, consequently, also $\nu(G)$). In the article \cite{BR} the authors establish some structural results concerning $\nu(G)$ when $G$ is a finite-by-nilpotent group. In the present paper we are concerned with the non-abelian tensor square of residually finite groups. 

According to the solution of the Restricted Burnside Problem (Zelmanov, \cite{ze1,ze2}) every residually finite group of finite exponent is locally finite. Another interesting result in this context, due to Shumyatsky \cite{shu3} states that if $G$ is a residually finite group satisfying a non-trivial identity and generated by a normal commutator-closed set of $p$-elements, then $G$ is locally finite. A subset $X$ of a group $G$ is called {\it commutator-closed} if $[x,y]\in X$ whenever $x,y\in X$. In a certain way, our results can be viewed as generalizations of some of the above results. 

\begin{thmA}
Let $p$ a prime. Let $m$ be a positive integer and $G$ a residually finite group satisfying some non-trivial identity $f \equiv~1$. Suppose that for every $x,y \in G$ there exists a positive integer $q=q(x,y)$ dividing $p^m$ such that the tensor $[x,y^{\vfi}]^q = 1$. Then the derived subgroup $\nu(G)'$ is locally finite. 
\end{thmA}

Let $G$ be a group. For elements $x,y$ of $G$ we define $[x,{}_1y]=[x,y]$ and $[x,{}_{i+1}y]=[[x,{}_{i}y],y]$ for $i\geq1$, where $[x,y] = x^{-1}y^{-1}xy$. An element $y\in G$ is called {\em left $n$-Engel} if for any $x\in G$ we have $[x,{}_{n}y]=1$. The group $G$ is called a {\em left $n$-Engel group} if $[x,{}_{n}y]=1$ for all $x,y \in G$.   

Another result that was deduced following the solution of the Restricted Burnside Problem is that any residually finite $n$-Engel group is locally nilpotent (J.\,S. Wilson, \cite{w}). Later, Shumyatsky proved that if $G$ is a residually finite group in which every commutator is left $n$-Engel in $G$, then the derived subgroup $G'$ is locally nilpotent \cite{shu1,shu2}. In the present paper we establish the following related result.

\begin{thmB} \label{ThmB}
Let $m,n$ be positive integers and $p$ a prime. Suppose that $G$ is a residually finite group in which for every $x,y \in G$ there exists a positive integer $q=q(x,y)$ dividing $p^m$ such that the element $[x,y^{\vfi}]^q$ is left $n$-Engel in $\nu(G)$. Then the non-abelian tensor square $G \otimes G$ is locally virtually nilpotent.   
\end{thmB}

A natural question arising in the context of Theorem B is whether the theorem remains valid when $q$ is allowed to be an arbitrary positive integer, rather than a $p$-power (see Proposition \ref{prop.dash} in Section 4 for details). 

The paper is organized as follows. Section 2 presents terminology and some preparatory results that are later used in the proofs of our main results. In Section 3 we describe some important ingredients of what is often called “Lie methods in group theory”. These are essential in the proof of Theorem B. Section 4 contains the proofs of the main theorems.

\section{Preliminary results}

The following basic properties are consequences of 
the defining relations of $\nu(G)$ and the commutator rules (see \cite[Section 2]{NR1} for more details). 

\begin{lem} 
\label{basic.nu}
The following relations hold in $\nu(G)$, for all 
$g, h, x, y \in G$.
\begin{itemize}
\item[$(i)$] $[g, h^{\varphi}]^{[x, y^{\varphi}]} = [g, h^{\varphi}]^{[x, 
y]}$; 
\item[$(ii)$] $[g, h^{\varphi}, x^{\varphi}] = [g, h, x^{\varphi}] = [g, 
h^{\varphi}, x] = [g^{\vfi}, h, x^{\vfi}] = [g^{\vfi}, h^{\vfi}, x] = 
[g^{\vfi}, 
h, x]$;
\item[$(iii)$] If $h \in G'$ (or if $g \in G'$) then $[g, h^{\varphi}][h, 
g^{\varphi}]=1$;
\item[$(iv)$] $[g, [h, x]^{\varphi}]= [[h, x], g^{\varphi}]^{-1}$;
\item[$(v)$] $[[g,h^{\vfi}],[x,y^{\vfi}]] = [[g,h],[x,y]^{\vfi}]$.
\end{itemize}
\end{lem}

\begin{lem} \label{lem.closed}
Let $G$ be a group. Then $X = \{[a,b^{\varphi}] \mid \ a,b \in G \}$ is a normal commutator-closed subset of $\nu(G)$.   
\end{lem}

\begin{proof}
By definition, $[a,b^{\vfi}]^{y^{\vfi}} = [a,b^{\vfi}]^y = [a^y,(b^y)^{\vfi}]$, for every $a,b,y \in G$. It follows that $X$ is a normal subset of $\nu(G)$. Now, it remains to prove that $X$ is a commutator-closed subset of $\nu(G)$. Choose arbitrarily elements $a,b,c,d \in G$ and consider the elements $[a,b^{\vfi}], [c,d^{\vfi}] \in X$. By Lemma \ref{basic.nu} $(v)$, $$[[a,b^{\varphi}],[c,d^{\varphi}]] = [[a,b],[c,d]^{\varphi}] \in X.$$ The result follows. 
\end{proof}

The epimorphism $\rho: \nu(G) \to G$, given by $g \mapsto g$, $h^{\vfi} \mapsto h$, induces the derived map $$\rho':[G,G^{\vfi}] \to G',$$ $[g,h^{\vfi}] \mapsto [g,h]$ for all $g,h \in G$. In the notation of \cite[Section 2]{NR2}, let $\mu(G)$ denote the kernel of $\rho'$. In particular, $$\dfrac{[G,G^{\vfi}]}{\mu(G)} \cong G'.$$ 

The next lemma is a particular case of Theorem 3.3 in \cite{NR1}.

\begin{lem} \label{thm.Rocco}
Let $G$ be a group. Then the derived subgroup $$\nu(G)' = ([G,G^{\vfi}] \cdot G') \cdot (G')^{\vfi},$$ where ``$\cdot$'' denotes an internal semi-direct product. 
\end{lem}

\begin{lem} \label{Engel}
Let $m,n$ be positive integers and $x,y$ be elements in a group $G$. If $[x,y^{\vfi}]^m$ is left $n$-Engel element in $\nu(G)$, then the power $[x,y]^m$ is left $n$-Engel in $G$.  
\end{lem}

\begin{proof}
By assumption, $[z,_{n} [x,y^{\vfi}]^m] = 1$ for every $z \in G$. It follows that $$1 = \rho([z,_{n}[x,y^{\vfi}]^m]) = [z,_{n}[x,y]^m],$$
for every $z \in G$, i.e., $[x,y]^m$ is left $n$-Engel in $G$, and the lemma follows. 
\end{proof}

The remainder of this section will be devoted to describe certain finiteness conditions for residually finite and locally graded groups. Recall that a group is {\it locally graded} if every non-trivial finitely generated subgroup has a proper subgroup of finite index. Interesting classes of groups (e.g. locally finite groups, locally nilpotent groups, residually finite groups) are locally graded. \\ 

We need the following result, due to Shumyatsky \cite{shu3}. 

\begin{thm} \label{thm.1} 
Let $G$ be a residually finite group satisfying some non-trivial identity $f \equiv~1$. Suppose $G$ is generated by a normal commutator-closed set $X$ of $p$-elements. Then $G$ is locally finite.
\end{thm}

Next, we extend this result to the class of locally graded groups.   

\begin{lem} \label{graded}
Let $p$ be a prime. Let $G$ be a locally graded group satisfying some non-trivial identity $f \equiv~1$. Suppose that $G$ is generated by finitely many elements of finite order and for every $x,y \in G$ there exists a $p$-power $q=q(x,y)$ such that $[x,y]^q = 1$. Then $G$ is finite.  
\end{lem}

\begin{proof}
Let $R$ be the finite residual of $G$, i.e., the intersection of all subgroups of finite index in $G$. If $R=1$, then $G$ is residually finite. Since the set of all commutators is a normal commutator-closed subset of $p$-elements, we have $G'$ is locally finite (Theorem \ref{thm.1}). Moreover, $G/G'$ is generated by finitely many elements of finite order. It follows that $G/G'$ is finite and so $G'$ is finitely generated. Consequently, $G$ is finite. So, we can assume that $R \neq 1$. By the same argument, $G/R$ is finite and thus $R$ is finitely generated. As $R$ is locally graded we have that $R$ contains a proper subgroup of finite index in $G$, which gives a contradiction. 
\end{proof}

It is easy to see that a quotient of a residually finite group need not be residually finite (see for instance \cite[6.1.9]{Rob}).  This includes the class of locally graded groups. However, the next result gives a sufficient condition for a quotient to be locally graded.

\begin{lem} (Longobardi, Maj, Smith, \cite{LMS}) \label{quotient}
Let $G$ be a locally graded group and $N$ a normal locally nilpotent subgroup of $G$. Then $G/N$ is locally graded.
\end{lem}

\section{On Lie Algebras Associated with Groups}

Let $L$ be a Lie algebra over a field $\mathbb{K}$.
We use the left normed notation: thus if
$l_1,l_2,\dots,l_n$ are elements of $L$, then
$$[l_1,l_2,\dots,l_n]=[\dots[[l_1,l_2],l_3],\dots,l_n].$$
We recall that an element $a\in L$ is called {\it ad-nilpotent} if there exists a positive integer $n$ such that $[x,{}_na]=0$ for all $x\in L$. When $n$ is the least integer with the above property then we say that $a$ is ad-nilpotent of index $n$. 

Let $X\subseteq L$ be any subset of $L$. By a commutator of elements in $X$,
we mean any element of $L$ that could be obtained from
elements of $X$ by means of repeated operation of
commutation with an arbitrary system of brackets
including the elements of $X$. Denote by $F$ the free
Lie algebra over $\mathbb{K}$ on countably many free
generators $x_1,x_2,\dots$. Let $f=f(x_1,x_2,
\dots,x_n)$ be a non-zero element of $F$. The algebra
$L$ is said to satisfy the identity $f=0$ if
$f(l_1,l_2,\dots,l_n)=0$ for any $l_1,l_2,\dots,l_n
\in L$. In this case we say that $L$ is PI. Now, we recall an important theorem of Zelmanov 
\cite[Theorem 3]{zelm} that has many applications in roup Theory.
\begin {theorem}\label{1} 
Let $L$ be a Lie algebra generated by $l_1,l_2,\dots,l_m$. Assume that $L$ is PI and that each commutator in the generators is ad-nilpotent. Then $L$ is nilpotent.
\end{theorem}

Let $G$ be a group and $p$ a prime. In what follows, $$D_i=D_i(G) = \displaystyle \prod_{jp^k \geqslant i} (\gamma_j(G))^{p^k} $$ 
denotes the $i$-th dimension subgroup of $G$ in characteristic
$p$. These subgroups form a central series of $G$
known as the {\it Zassenhaus-Jennings-Lazard series} (this can be found in Shumyatsky \cite[Section 2]{shu2}). Set $L(G)=\bigoplus D_i/D_{i+1}$. 
Then $L(G)$ can naturally be viewed as a Lie algebra 
over the field ${\mathbb F}_p$ with $p$ elements. The subalgebra of $L(G)$ generated by $D_1/D_2$ will be 
denoted by $L_p(G)$. The nilpotency of $L_p(G)$ has strong influence in the structure of a finitely generated group $G$. The following result is due to Lazard \cite{l2}.

\begin{thm}\label{3} 
Let $G$ be a finitely generated pro-$p$
group. If $L_p(G)$ is nilpotent, then $G$ is
$p$-adic analytic.
\end{thm}

Let $x\in G$, and let $i=i(x)$ be the largest integer such
that $x\in D_i$. We denote by ${\tilde x}$ the 
element $xD_{i+1}\in L(G)$. Now, we can state one condition for ${\tilde x}$ to be ad-nilpotent.

\begin{lem} \label{4}(Lazard, \cite[page 131]{la}) 
For any $x\in G$ we have
$(ad\,{\tilde x})^q=ad\,(\widetilde {x^q})$.
\end{lem}

\begin{rem}
We note that $q$ in Lemma \ref{4} does not need to be a 
$p$-power. In fact, is easy to see that if $p^s$ is the 
maximal $p$-power dividing $q$, then $\tilde x$ is 
ad-nilpotent of index at most $p^s$.
\end{rem}

The following result is an immediate corollary of \cite[Theorem 1]{wize}.
\begin{lem}\label{2}
Let $G$ be any group satisfying a group-law. Then $L(G)$ is $PI$.
\end{lem}

For a deeper discussion of applications of Lie methods to group theory we refer the reader to \cite{shu2}.

\section{Proof of the main results}

By a well-known theorem of Schmidt \cite[14.3.1]{Rob}, the class of locally finite groups is closed with respect to extension. By Lemma \ref{thm.Rocco}, $\nu(G)' = ([G,G^{\vfi}] \cdot G') \cdot (G')^{\vfi}$. From this and Proposition \ref{prop.g} we conclude that the non-abelian tensor square $G \otimes G$ is locally finite if and only if $\nu(G)'$ is locally finite. The next result will be needed in the proof of Theorem A.

\begin{prop} \label{prop.1}
Let $G$ be a group in which $G'$ is locally finite. The following properties are equivalents.
\begin{itemize}
 \item[$(a)$] $[G,G^{\vfi}]$ is locally finite;
 \item[$(b)$] $\nu(G)'$ is locally finite;
 \item[$(c)$] For every $x,y \in G$ there exists a positive integer $m=m(x,y)$ such that $[x,y^{\vfi}]^m = 1$.
\end{itemize}
\end{prop}

\begin{proof}
(a) $\Rightarrow$ (b): It is sufficient to see that $\nu(G)' = ([G,G^{\vfi}] \cdot G') \cdot (G')^{\vfi}$ (Theorem \ref{thm.Rocco}). \\ 

(b) $\Rightarrow$ (c): Since $\nu(G)'$ is locally finite, it follows that every element in $\nu(G)'$ has finite order. In particular, for every $x,y \in G$ the element $[x,y^{\vfi}]$ has finite order. \\

(c) $\Rightarrow$ (a): Let us first prove that $[G,G^{\vfi}]$ is locally finite. Let $W$ be a finitely generated subgroup of $[G,G^{\vfi}]$. Since the factor group $[G,G^{\vfi}]/\mu(G)$ is isomorphic to $G'$, it follows that $W$ is central-by-finite. By Schur's Lemma \cite[10.1.4]{Rob}, the derived subgroup $W'$ is finite. Since $W/W'$ is an abelian group generated by finite many elements of finite order, it follows that $W/W'$ is also finite. Thus $[G,G^{\vfi}]$ is locally finite. Now, Lemma \ref{thm.Rocco} shows that $\nu(G)' = ([G,G^{\vfi}] \cdot G') \cdot (G')^{\vfi}$. Therefore $\nu(G)'$ is locally finite, by \cite[14.3.1]{Rob}. 
\end{proof}

We are now in a position to prove Theorem A.

\begin{thmA}
Let $p$ be a prime and $G$ a residually finite group satisfying some non-trivial identity $f \equiv~1$. Suppose that for every $x,y \in G$ there exists a $p$-power $q=q(x,y)$ such that $[x,y^{\vfi}]^q = 1$. Then the derived subgroup $\nu(G)'$ is locally finite. 
\end{thmA}

\begin{proof}
The proof is completed by showing that $G'$ is locally finite (Proposition \ref{prop.1}). Set $$X = \{ \ [a,b^{\varphi}]; \ a,b\in G \}.$$ 

By Lemma \ref{lem.closed}, $X$ is a normal commutator-closed subset of $p$-elements in $\nu(G)$. Since $[G,G^{\vfi}]/\mu(G) \cong G'$, we conclude that the derived subgroup $G'$ is residually finite satisfying the identity $f \equiv~1$. In the same manner we can see that for every $x,y$ there exists a $p$-power $q=q(x,y)$ such that $[x,y]^q = 1$. Theorem \ref{thm.1} now implies shows that $G'$ is locally finite. The proof is complete. 
\end{proof}

\begin{rem}
For more details concerning residually finite groups in which the derived subgroup is locally finite see \cite{shu,shu3}.
\end{rem}

{\noindent} Now we will deal with Theorem B: \emph{Let $m,n$ be positive integers and $p$ a prime. Suppose that $G$ is a residually finite group in which for every $x,y \in G$ there exists a $p$-power $q=q(x,y)$ dividing $p^m$ such that $[x,y^{\vfi}]^q$ is left $n$-Engel in $\nu(G)$. Then the non-abelian tensor square $G \otimes G$ is locally virtually nilpotent.}   \\

We denote by $\mathcal{N}$ the class of all finite nilpotent groups. The following result is a straightforward corollary of \cite[Lemma 2.1]{w}.

\begin{lem}\label{Wilson}
Let $G$ be a finitely generated residually-$\mathcal{N}$ group. For each prime $p$, let $R_p$ be the intersection of all normal subgroups of $G$ of finite $p$-power index. If $G/R_p$ is nilpotent for each $p$, then $G$ is nilpotent.
\end{lem}

The next result will be helpful in the proof of Theorem B.

\begin{prop} \label{prop.dash}
Let $m,n$ be positive integers and $p$ a prime. Suppose that $G$ is a residually finite group in which for every $x,y \in G$ there exists a $p$-power $q=q(x,y)$ dividing $p^m$ such that $[x,y]^q$ is left $n$-Engel. Then $(G')^{p^m}$ is locally nilpotent. Moreover, the derived subgroup $G'$ is locally virtually nilpotent. 
\end{prop}

\begin{proof}
Let $K$ be the derived subgroup of $G$. Choose an arbitrary commutator $x \in G$ and let $q=q(x)$ be a positive integer $q=q(x)$ dividing $p^m$ such that $x^q$ is left $n$-Engel. It suffices to prove that the normal closure of $x^q$ in $G$, $\langle (x^q)^h \ \vert \ h \in G  \rangle$, is locally nilpotent. In particular, $x^q \in HP(G)$. Let $b_1, \ldots,b_t$ be finitely many elements in $G$. Let $h_i = (x^q)^{b_i}$, $i=1, \ldots,t$ and $H = \langle h_1, \ldots, h_t\rangle$. We only need to show that $H$ is soluble (Gruenberg's Theorem \cite[12.3.3]{Rob}). As a consequence of Lemma \ref{Wilson}, we can assume that $H$ is residually-$p$ for some prime $p$. Let $L=L_p(H)$ be the Lie algebra associated with the Zassenhaus-Jennings-Lazard series $$H=D_1\geq D_2\geq \cdots$$ of $H$. Then $L$ is generated by $\tilde{h}_i=h_i D_2$, $i=1,2,\dots,t$. Let $\tilde{h}$ be any Lie-commutator in $\tilde{h}_i$ and $h$ be the group-commutator in $h_i$ having the same system of brackets as $\tilde{h}$. Since for any group commutator $h$ in $h_1\dots,h_t$ there exists a positive integer $q$ dividing $p^m$ such that $h^q$ is left $n$-Engel, Lemma \ref{4} shows that any Lie commutator in $\tilde h_1\dots,\tilde h_t$ is ad-nilpotent. On the other hand, for every commutator $x = [a_1,b_1]$ there exists a positive integer $q = q(x)$ dividing $p^m$ such that $x^q$ is left $n$-Engel, then $G$ satisfies the identity
$$f = [z, {}_{n}[x_1,y_1], {}_{n} [x_1,y_1]^p, {}_{n} [x_1,y_1]^{p^2}, {}_{n} [x_1,y_1]^{p^3}, \ldots, {}_{n} [x_1,y_1]^{p^m}]\equiv 1$$
and therefore, by Lemma \ref{2}, $L$ satisfies some non-trivial polynomial identity. Now Zelmanov's Theorem \ref{1} implies that $L$ is nilpotent. Let $\hat{H}$ denote the pro-$p$ completion of $H$. Then $L_p(\hat{H})=L$ is nilpotent and $\hat{H}$ is a $p$-adic analytic group by Theorem \ref{3}. Clearly $H$ cannot have a free subgroup of rank 2 and so, by Tits' Alternative \cite{tits}, $H$ is virtually soluble. As $H$ is residually-$p$ we have $H$ is soluble. 

Choose arbitrarily finitely many commutators $[a_1,b_1], \ldots, [a_t,b_t]$ and let $M = \langle [a_1,b_1], \ldots, [a_t,b_t] \rangle$. Since $MK^{p^m}$ is residually finite and $K^{p^m}$ is locally graded, we conclude that $(MK^{p^m})/K^{p^m}$ is also locally graded (Lemma \ref{quotient}). Now, looking at the quotient $(MK^{p^m})/K^{p^m}$, the claim is immediate from Lemma \ref{graded}. As $M/(M \cap K^{p^m})$ is finite we have $M \cap K^{p^m}$ finitely generated. In particular, $M \cap K^{p^m}$ is nilpotent and $M$ is virtually nilpotent. This completes the proof.
\end{proof}
\begin{rem}
The above proposition is no longer valid if the assumption of residual finiteness of $G$ is dropped. In \cite{DK} G.\,S. Deryabina and P.\,A. Kozhevnikov showed that there exists an integer $N \geqslant 1$ such that for every odd number $n > N$ there is a group $G$ with commutator subgroup $G'$ not locally finite and satisfying the identity $$f = [x,y]^n \equiv~1.$$ In particular, $G'$ cannot be locally virtually nilpotent. For more details concerning groups in which certain powers are left Engel elements see \cite{B,BS,BSTT,STT}.
\end{rem}

The proof of Theorem B is now easy. 

\begin{proof}[Proof of Theorem B]
Let $M$ be a finitely generated subgroup of $[G,G^{\vfi}]$. Clearly, there exist finitely many elements $a_1,\ldots,a_s$, $b_1, \ldots,b_s \in G$ such that $$M  \leqslant \langle [a_i,b_i^{\varphi}]  \mid \ i = 1,\ldots,s  \rangle = H.$$
It suffices to prove that $H$ is virtually nilpotent. Since for every $a,b \in G$ there exists a $p$-power $q=q(a,b)$ dividing $p^m$ such that $[a,b^{\vfi}]^q$ is left $n$-Engel, we deduce that the power $[a,b]^q$ is left $n$-Engel (Lemma \ref{Engel}). That $G'$ is locally virtually nilpotent follows from Proposition \ref{prop.dash}. Therefore, since $\mu(G)$ is the kernel of the derived map $\rho'$, it follows that $H/(H \cap \mu(G))$ is virtually nilpotent. As $\mu(G)$ is a central subgroup of $\nu(G)$ we have $H$ virtually nilpotent, which proves the theorem.  
\end{proof}

We conclude this paper by formulating an open problem related to the results described here. \\ 

{\noindent}{\bf Problem 1.} {\em Let $m,n$ be positive integers. Suppose that $G$ is a residually finite group in which for every $x,y \in G$ there exists a positive integer $q=q(x,y)$ dividing $m$ such that  $[x,y^{\vfi}]^q$ is left $n$-Engel in $\nu(G)$. Is it true that $G \otimes G$ is locally virtually nilpotent?} \\

\end{document}